\documentclass{amsart}
\usepackage{amssymb
,amsthm,dirtytalk
,amsmath
,amscd,mathrsfs
,mathtools,enumerate
}
\usepackage{hyperref}
\usepackage[all]{xy}
\usepackage[top=30truemm,bottom=30truemm,left=25truemm,right=25truemm]{geometry}

\newcommand{\GL}{\mathrm{GL}}

\newcommand{\Hom}{\mathrm{Hom}}

\newtheorem{thm}{Theorem}[section]

\newtheorem{lem}[thm]{Lemma}
\newtheorem{prop}[thm]{Proposition}
\newtheorem{coro}[thm]{Corollary}

\theoremstyle{remark}
\newtheorem{rem}[thm]{Remark}

\theoremstyle{definition}
\newtheorem{defn}[thm]{Definition}

\numberwithin{equation}{section}

\makeatletter
\def\iddots{\mathinner{\mkern1mu\raise\p@
	\hbox{.}\mkern2mu\raise4\p@\hbox{.}\mkern2mu
	\raise7\p@\vbox{\kern7\p@\hbox{.}}\mkern1mu}}
\def\adots{\mathinner{\mkern2mu\raise\p@\hbox{.}
 \mkern2mu\raise4\p@\hbox{.}\mkern1mu
 \raise7\p@\vbox{\kern7\p@\hbox{.}}\mkern1mu}}
\makeatother

\allowdisplaybreaks
\title{The generalized linear periods}
\author{Hengfei LU}
\address{Department of Mathematics, University of Vienna, Oskar-Morgenstern-Platz 1, Wien 1090, Austria}
\email{hengfei.lu@univie.ac.at}

\begin{document}
%\maketitle

\begin{abstract} Let $F$ be a  local field of characteristic zero. Let $\mu$ be a good character of 
	$\GL_p(F)\times\GL_{p+1}(F)$.
We study the generalized linear period problem for the pair $(G,H_{p,p+1})=(\GL_{2p+1}(F),\GL_{p}(F)\times\GL_{p+1}(F))$ and we prove that any  bi-$(H_{p,p+1},\mu)$-equivariant tempered generalized function on $G$ is invariant under the matrix transpose. We also show that any $P\cap H_{p,p+1}$-invariant 
linear functional on an $H_{p,p+1}$-distinguished irreducible smooth  representation of $G$ is also $H_{p,p+1}$-invariant if $F$ is nonarchimedean, where $P$ is the standard mirabolic subgroup of $G$ consisting of matrices with last row vector $(0,\cdots,0,1)$.
\end{abstract}
\maketitle
\tableofcontents
%\linenumbers
\section{Introduction}
Let $F$ be a local field of characteristic zero. Let $p,q,n$ be positive integers and $n=p+q$. Let $\theta_{p,q}$ be the involution defined on $\GL_n(F)$ given by
\[\theta_{p,q}(g)=\omega_{p,q}\cdot g\cdot\omega_{p,q} \]
for $g\in\GL_n(F)$ where
 $\omega_{p,q}=\begin{pmatrix}
\mathbf{1}_p\\&-\mathbf{1}_q
\end{pmatrix}$ and $\mathbf{1}_p$ (resp. $\mathbf{1}_q$) is the identity matrix in the $p\times p$ (resp. $q\times q$) matrix space $Mat_{p,p}(F)$ (resp. $Mat_{q,q}(F)$). Let $H_{p,q}$ be the fixed points of $\theta_{p,q}$ in $\GL_n(F)$. Then $H_{p,q}\cong \GL_p(F)\times\GL_q(F)$.  It is well known that the pair $(\GL_n(F),\GL_p(F)\times\GL_q(F) )$ satisfies the Gelfand-Kazhdan criterion \cite[\S7]{dima2009duke} with respect to the inverse map; see \cite{jacquet1996linear} for the non-archimedean case and \cite[Theorem 7.1.3]{dima2009duke} for the archimedean case. It implies that $\dim \Hom_{H_{p,q}}(\pi,\mathbb{C})\leq 1 $ for all irreducible admissible smooth representation $\pi$ of $\GL_n(F)$.
 Jacquet-Rallis \cite{jacquet1996linear} proved that if $\dim\Hom_{H_{p,q}}(\pi,\mathbb{C})=1$ and $F$ is non-archimedean, then $\pi\cong\pi^\vee$ where $\pi^\vee$ denotes the representation of $\GL_n(F)$ contragredient to $\pi$. When $p=q$, it  is closely related to the Shalika period problems (see \cite{sun2020}). Furthermore, Friedberg-Jacquet \cite{FJ1993} have studied the relation between the linear period of $\pi$ and the exterior square $L$-function $L(s,\pi,\Lambda^2)$.

This paper studies the twisted version of the linear period.
We say that a character of $F^\times$ is pseudo-algebraic if it has the form
\[t\mapsto\begin{cases}
1,&\mbox{ if }F\mbox{ is nonarchimedean,}\\
t^m,&\mbox{ if }F=\mathbb{R},\\
\iota(t)^m\iota'(t)^{m'},&\mbox{ if }F\cong\mathbb{C},
\end{cases} \]
where $m$ and $m'$ are non-negative integers and $\iota$ and $\iota'$ are two
distinct topological isomorphisms from $F$ to $\mathbb{C}$.
Let $\mu_F$ be a character of $F^\times$ and $\mu_F\circ\det$ be a character of $\GL_p(F)$. Let $\mu_F\circ\det \otimes\mathbb{C}$ be a character of $H_{p,q}$,  denoted by $\mu$.
 We say that $\mu$ is a \textbf{good} character of $H_{p,q}$ if
\begin{itemize}
	\item $\mu_F^{2r}|-|^{-s}$ is not pseudo-algebraic
\end{itemize}
for all $r\in\{\pm1,\pm2,\cdots,\pm p\}$ and all $s\in\{1,2,\cdots, 2p^2\}$. (See \cite{sun2020} for more details.)
From now on, we assume that $q=p+1$  throughout this paper, unless otherwise specified. One of
the main results in this paper is the following:
\begin{thm}\label{thm:A}
Suppose that $n=2p+1$.	Let $f$ be  a tempered generalized function on $\GL_n(F)$. If for every $h\in H_{p,p+1}$,
	\[f(hx)=f(xh)=\mu(h)f(x) \]
for $x\in\GL_n(F)$ and any good character $\mu$,	as generalized functions on $\GL_n(F)$, then 
\[f(x)=f(x^t). \]
\end{thm}
Here and as usual, a superscript \say{t} indicates the transpose of a matrix. Then the pair $(\GL_{2p+1}(F),\GL_p(F)\times\GL_{p+1}(F) )$ satisfies the generalized Gelfand-Kazhdan criterion (see \cite[Theorem 2.3]{sunzhu2011}) with respect to the matrix transpose, which implies that
\[\dim\Hom_{H_{p,p+1} }(\pi,\mu)\cdot\dim\Hom_{H_{p,p+1}}(\pi^\vee,\mu^{-1}) \leq 1 \]
for any irreducible admissible smooth representation $\pi$ of $\GL_{2p+1}(F)$ and any good character $\mu$ of $H_{p,p+1}$.
 The analogue for the pair $(\GL_{2p}(F),\GL_p(F)\times\GL_p(F) )$ has been proved by Chen-Sun in \cite{sun2020}. We will use a similar idea appearing in \cite{sun2020} to prove Theorem \ref{thm:A}.

Define $I_{p,p+1}:=Mat_{p,p+1}(F)\oplus Mat_{p+1,p}(F)$ and $\mathcal{N}_{p,p+1}:=\{(x,y)\in I_{p,p+1}|(xy)^p=0 \}$.
Denote by $\mathscr{C}_{\mathcal{N}_{p,p+1}}(I_{p,p+1})$ the space consisting of tempered generalized functions on $I_{p,p+1}$ supported on $\mathcal{N}_{p,p+1}$.
By linearization, Theorem \ref{thm:A} is reduced to the following theorem.
\begin{thm}\label{tran:M}
	Let $f$ be a tempered generalized function on $I_{p,p+1}$ supported on the nilpotent cone $\mathcal{N}_{p,p+1}$  such that for $h=\begin{pmatrix}
	a\\& b
	\end{pmatrix}\in H_{p,p+1}$,
	\[f(axb^{-1},bya^{-1})=f(x,y) \]
holds	for any $(x,y)\in I_{p,p+1}$. Then $f(x,y)=f(y^t,x^t)$.
\end{thm} 
There is a brief introduction to the proof of Theorem \ref{tran:M}. We will regard the $n$-dimensional vector space as a graded $\mathfrak{sl}_2(F)$-module. Chen-Sun \cite{sun2020} used the graded modules and Fourier transform to prove that there does not exist any $H_{p,p}$-invariant generalized function $f$ on  $Mat_{p,p}(F)\times Mat_{p,p}(F)$ such that both $f$ and its Fourier transform $\mathfrak{F}(f)$ are supported on the nilpotent cone of $Mat_{p,p}(F)\times Mat_{p,p}(F)$. However, there may exist $H_{p,p+1}$-invariant generalized functions $f_0$ on $I_{p,p+1}$ such that both $f_0$ and its Fourier transform $\mathfrak{F}(f_0)$ are supported on the orbit $H_{p,p+1}\mathbf{e}$ (the regular nilpotent orbit), where $\mathbf{e}^{2p}\neq 0$. There is a key observation due to Dmitry Gourevitch that $\mathbf{e}^t\in H_{p,p+1}\mathbf{e}$ and so if $f_0\in \mathscr{C}_{\mathcal{N}_{p,p+1}}(I_{p,p+1})^{\tilde{H}_{p,p+1},\chi}$ (see Theorem \ref{vanishing:I}) then $f_0=0$. Therefore,
\[\mathscr{C}_{\mathcal{N}_{p,p+1}}(I_{p,p+1})^{\tilde{H}_{p,p+1},\chi}=0 \]
i.e. Theorem \ref{tran:M} holds. (All the techniques in this paper work for the pair $(\GL_{2p+1}(F),\GL_{p+1}(F)\times\GL_p(F) )$ as well. But they do not work for the pair $(\GL_{2p+2}(F),\GL_p(F)\times\GL_{p+2}(F))$ because Proposition \ref{observation} fails; see Remark \ref{rem:fail}.) 
In fact, we will prove a stronger result that any $H_{p,p}$-invaraint generalized function on $I_{p,p+1}$ is also invariant under transposition, where $H_{p,p}$ is a proper subgroup of $H_{p,p+1}$. (See the proof of Theorem \ref{mir}.)

In a similar way, we can prove the following.
\begin{thm}\label{p=1}
	Let $f$ be a tempered generalized function on $\GL_n(F)$. Let $\mu_F$ be any character (not necessarily good) of $F^\times$. If for every $h\in H_{1,n-1}$,
	\[f(hx)=f(xh)=\mu_F(a)f(x) \]
	for $x\in\GL_n(F)$ and $h=\begin{pmatrix}
	a\\&b
	\end{pmatrix}$ with $a\in F^\times,b\in\GL_{n-1}(F)$, as generalized functions on $\GL_n(F)$, then
	\[f(x)=f(x^t). \]
\end{thm}
\begin{rem}
	In \cite{eitan2008comp}, Aizenbud-Gourevitch-Sayag use a different method to obtain a stronger result that any bi-$\GL_{n-1}(F)$-invariant generalized function on $\GL_n(F)$ is invariant with repect to transposition for any local field $F$ when $\mu_F$ is trivial.
	Here $\GL_{n-1}(F) $ is regarded as a proper subgroup of $H_{1,n-1}$. 
	Inspired by their results in \cite{eitan2008comp}, we have Theorem \ref{mir}.
	Furthurmore, there is a much stronger result  that any invariant distribution on $\GL_{n+1}(F)$ under the adjoint action of $\GL_{n}(F)$ is invariant  with respect to transposition (see  \cite{AGRS} for the non-archimedean case and \cite{AG2009selecta,sunzhu2012annals} for the archimedean case), which recently has been extended to the case when $F$ is of positive characteristic different from $2$ (see \cite{mezer2019multiplicity}).
\end{rem}

Finally, we give one application to the vanishing of certain distributions which are equivariant under 
transposition when $F$ is non-archimedean. 
%i.e. if $f^t=-f$ for some generalized functions which are invariant under a group action, then  $f$ is zero. 
More precisely, we have shown that any $H_{p,p}$-invariant generalized function on $I_{p,p+1}$ supported on $\mathcal{N}_{p,p+1}$ is also invariant under transposition, which implies that
%which gives a new and shorter proof to \cite[Theorem 4.2]{maxmirabolic} that
any $P\cap H_{p,p+1}$-invariant linear functional on an $H_{p,p+1}$-distinguished irreducible smooth  representation of $\GL_{2p+1}(F)$ is also $H_{p,p+1}$-invariant, which extends the result of Maxim Gurevich in \cite{maxmirabolic},
 where $P$ is a standard mirabolic subgroup of $\GL_{2p+1}(F)$. (See Theorem \ref{mir}.) This is the original motivation of writing this paper.

The paper is organized as follows. In \S2, we introduce some notation about the algebraic geometry. Then we will use Chen-Sun's method to prove Theorem \ref{tran:M} in \S3. The proof of Theorem \ref{thm:A} will be given 
in \S4, which heavily depends on the results of Aizenbud-Gourevitch (see Theorem \ref{thm2.1}). We shall prove Theorem \ref{p=1} in \S5. The last section studies the role of the mirabolic subgroup acting on the symmetric variety $\GL_{n}(F)/\GL_{p}(F)\times\GL_{n-p}(F)$ following Maxim Gurevich in \cite{maxmirabolic}.

\section{Preliminaries and notation}

Let $X$ be an $\ell$-space (i.e. locally compact totally disconnected topological space) if $F$ is non-archimedean 
or a Nash manifold (see \cite[\S2.3]{dima2009duke}) if $F$ is archimedean. Let $\mathscr{C}(X)$ denote the space of tempered generalized functions on $X$. Let a reductive group $G(F)$
act on an affine variety $X$. Let $x\in X$ such that its orbit $G(F)x$ is closed in $X$. We denote the normal bundle by $N_{G(F)x}^X$ and denote its fiber (the normal space) at the point $x$ by $N_{G(F)x,x}^X$. Let $$G_x:=\{g\in G(F)|gx=x \}$$ be the stalizer subgroup of $x$. Let $\chi$ be a character of $G(F)$. Denote by $\mathscr{C}(X)^{G(F),\chi}$ the subspace in $\mathscr{C}(X)$ consisting of those tempered generalized functions $f$ satisfying
\[g\cdot f=\chi(g)f \] 
for all $g\in G(F)$. If $\chi$ is trivial, then it will be denoted by $\mathscr{C}(X)^{G(F)}$.
\begin{thm}\cite[Theorem 3.1.1]{dima2009duke}
	Let $G(F)$ act on a smooth affine variety $X$. Let $\chi$ be a character of $G(F)$. Suppose that for any closed orbit $G(F)x$ in $X$, we have
	\[\mathscr{C}(N^X_{G(F)x,x})^{G_x,\chi}=0. \]
	Then \[\mathscr{C}(X)^{G(F),\chi}=0. \]\label{thm:duke}
\end{thm}
 If $V$ is a finite
dimensional representation of $G(F)$, then we denote the nilpotent cone in $V$ by
\[\Gamma(V):=\{ x\in V|\overline{G(F)x}\owns 0 \}. \] 
%In general, we define $S_y:=\{x\in X|\overline{G(F)x}\owns y \}$.
Let $Q_G(V):=V/V^G$. There is a canonical embedding $Q_G(V)\hookrightarrow V$ (see \cite[Notation 2.3.10]{dima2009duke}). Set $R_G(V):=Q(V)\setminus \Gamma(V)$. There is a stronger version of Theorem \ref{thm:duke}. 
\begin{thm}\cite[Corollary 3.2.2]{dima2009duke}.\label{thm2.1}
	Let $X$ be a smooth affine variety. Let $G(F)$ act on $X$. Let $K\subset G(F)$ be an open subgroup and let $\chi$ be a character of $K$. Suppose that for any closed orbit $G(F)x$ such that
	\[\mathscr{C}(R_{G_x}(N_{G(F)x,x}^X))^{K_x,\chi}=0 \]
	we have \[\mathscr{C} (Q_{G_x}(N^X_{G(F)x,x}))^{K_x,\chi}=0. \]
	Then $\mathscr{C}(X)^{K,\chi}=0$.
\end{thm}

\section{A vanishing result of generalized functions}
In this section, we shall use $q$ to denote $p+1$.
Let 
$$I_{p,q}=Mat_{p,q}(F)\oplus Mat_{q,p}(F)=\Big\{\begin{pmatrix}
0&x\\ y&0
\end{pmatrix}:x\in Mat_{p,q}(F),y\in Mat_{q,p}(F) \Big\}\subset \mathfrak{gl}_n(F).$$
Denote by 
\[\mathcal{N}_{p,q}:=\{(x,y)\in I_{p,q}|x y \mbox{ is a nilpotent matrix in }Mat_{p,p}(F) \} \]
the nilpotent cone in $I_{p,q}$.
 Denote $\tilde{H}_{p,q}:=H_{p,q}\rtimes\langle\sigma\rangle$ where $\sigma$ acts on $H_{p,q}$ by the involution 
$$\begin{pmatrix}
a\\&b
\end{pmatrix}\mapsto\begin{pmatrix}
(a^{-1})^t\\&(b^{-1})^t
\end{pmatrix}.$$
 The group $\tilde{H}_{p,q}$ acts on $I_{p,q}$ by
\[\begin{pmatrix}
a\\&b
\end{pmatrix}\cdot (x,y)=(axb^{-1},bya^{-1}) \]
and $$\sigma\cdot (x,y)=(y^t,x^t)$$
 for $(x,y)\in I_{p,q}$. Let $\chi$ be the sign character of $\tilde{H}_{p,q}$, i.e.
$\chi|_{H_{p,q}}$ is trivial and
\[\chi(\sigma)=-1. \]
Denoted by $\mathscr{C}_{\mathcal{N}_{p,q}}(I_{p,q})$ the space of tempered generalized functions on $I_{p,q}$ supported on $\mathcal{N}_{p,q}$. Set
\[\mathscr{C}_{\mathcal{N}_{p,q}}(I_{p,q})^{\tilde{H}_{p,q},\chi}:=\{f\in\mathscr{C}_{\mathcal{N}_{p,q}}(I_{p,q})|g\cdot f=\chi(g)f\mbox{  for all  }g\in\tilde{H}_{p,q} \}. \]
\begin{thm}\label{vanishing:I}
	We have $\mathscr{C}_{\mathcal{N}_{p,q}}(I_{p,q})^{\tilde{H}_{p,q},\chi}=0$.
\end{thm}
The rest part of this section is devoted to proving Theorem \ref{vanishing:I}. Then Theorem \ref{tran:M} follows from Theorem \ref{vanishing:I} directly by definition.

Define a non-degenerate symmetric $F$-bilinear form on $\mathfrak{gl}_n(F)$ by
\[\langle z, w\rangle_{\mathfrak{gl}_n(F)}:=\mbox{the trace of }zw\mbox{ as a }F\mbox{-linear operator}. \]
Note that the restriction of this bilinear form on $I_{p,q}$ is still non-degenerate. Fix a non-trivial unitary character $\psi$ of $F$. Denote by
\[\mathfrak{F}:\mathscr{C}(I_{p,q})\longrightarrow \mathscr{C}(I_{p,q}) \]
the Fourier transform which is normalized such that for every Schwartz function $\varphi$ on $I_{p,q}$,
\[\mathfrak{F}(\varphi)(z)=\int_{I_{p,q}}\varphi(w)\psi(\langle z,w \rangle_{\mathfrak{gl}_n(F)})d w  \]
for $z\in I_{p,q}$, where $d w$ is the self-dual Haar measure on $I_{p,q}$. If $I_{p,q}$ can be decomposed into a direct sum of two quadratic subspaces $U_1\oplus U_2$ such that each $U_i$ is non-degenerate with respect to $\langle-,-\rangle|_{U_i}$, then we may define the partial Fourier transform
\[\mathfrak{F}_{U_1}(\varphi)(x,y)=\int_{U_1}\varphi(z,y)\psi( \langle x,z\rangle|_{U_1}) dz \]
for $x\in U_1,y\in U_2$ and $\varphi\in \mathscr{C}(U_1\oplus U_2)$. Similarly for $\mathfrak{F}_{U_2}(\varphi)$.
 It is clear that the Fourier transform $\mathfrak{F}$ intertwines the action of $\tilde{H}_{p,q}$. Thus we have the following lemma.
\begin{lem}
	The Fourier transform $\mathfrak{F}$ preserves the space $\mathscr{C}_{\mathcal{N}_{p,q}}(I_{p,q})^{\tilde{H}_{p,q},\chi}$.\label{lem:intertwin}
\end{lem}

\subsection{Reduction within the null cone}
Recall
\[\mathcal{N}_{p,q}=\{(x,y)\in I_{p,q}|x y \mbox{ is a nilpotent matrix in }Mat_{p,p}(F) \}. \]
%the nilpotent cone in $I_{p,q}$. 
Let $\mathcal{O}$ be an $H_{p,q}$-orbit in $\mathcal{N}_{p,q}$. Recall that every $\mathbf{e}\in\mathcal{O}$ can be extended to a graded $\mathfrak{sl}_2$-triple $\{\mathbf{h},\mathbf{e},\mathbf{f} \}$ (see \cite[Proposition 4]{kostant}) in the sense that 
\begin{equation}\label{sl}
[\mathbf{h},\mathbf{e}]=2\mathbf{e},~[\mathbf{h},\mathbf{f}]=-2\mathbf{f}\mbox{ and }[\mathbf{e},\mathbf{f}]=\mathbf{h} \end{equation}
where $\mathbf{f}\in\mathcal{N}_{p,q}$ and $\mathbf{h}\in\mathfrak{h}_{p,q}$, where $\mathfrak{h}_{p,q}=\mathfrak{gl}_p(F)\oplus\mathfrak{gl}_q(F)$ is the Lie algebra of $H_{p,q}$. Let $I_{p,q}^{\mathbf{f}}$ denote the elements in $I_{p,q}$ annihilated by $\mathbf{f}$ under the adjoint action of the $\mathfrak{sl}_2$-triple $\{\mathbf{h},\mathbf{e},\mathbf{f} \}$ on $I_{p,q}\subset\mathfrak{gl}_n(F)$. Then 
\[I_{p,q}=[\mathfrak{h}_{p,q},\mathbf{e}]+I_{p,q}^{\mathbf{f}}. \]
 Following \cite[Proposition 3.9]{sun2020},
we shall prove the following proposition in this subsection.
\begin{prop}\label{fourier}
	Let $f$ be a $H_{p,q}$-invariant tempered  generalized function on $I_{p,q}$ such that 
	 $f$  and its Fourier transforms $\mathfrak{F}(f)$  are all supported on an orbit $\mathcal{O}=H_{p,q}\mathbf{e}\subset \mathcal{N}_{p,q}$. If $tr(2-\mathbf{h})|_{I_{p,q}^\mathbf{f}}\neq 2pq$ and $F$ is non-archimedean, then $f=0$.
	 If $F$ is archimedean and $tr(2-\mathbf{h})|_{I_{p,q}^\mathbf{f}}$ is not of the form $2pq-e$ with $e\geq0$, then $f=0$. 
\end{prop}

Denote by $\mathscr{C}_\mathcal{O}(I_{p,q})$ the space of tempered generalized functions on $I_{p,q}\setminus (\partial \mathcal{O})$ with support in $\mathcal{O}$, where $\partial \mathcal{O}$ is the complement of $\mathcal{O}$ in its closure in $I_{p,q}$. (See \cite[Notation 2.5.3]{dima2009duke}.) We will  use similar notation without further explaination.

Let $F^\times$ act on $\mathscr{C}(I_{p,q})$ by
\[(t\cdot f)(x,y)=f(t^{-1}x,t^{-1}y) \]
for $t\in F^\times$, $(x,y)\in I_{p,q}$ and $f\in\mathscr{C}(I_{p,q})$. The orbit $\mathcal{O}$ is invariant under dilation and so $F^\times$ acts on $\mathscr{C}_{\mathcal{O}}(I_{p,q})^{H_{p,q}}$ as well.
\begin{lem}\cite[Lemma 3.13]{sun2020}
	Let $\eta:F^\times\rightarrow\mathbb{C}^\times$ be an eigenvector for the action of $F^\times$ on $\mathscr{C}_{\mathcal{O}}(I_{p,q})^{H_{p,q}}$. Then
	$\eta^2=|-|^{tr(2-\mathbf{h})|_{I_{p,q}^\mathbf{f}}}\kappa$ for some pseudo-algebraic character $\kappa$ of $F^\times$.\label{key:lem}
\end{lem}

Let $Q$ be a quadratic form on $I_{p,q}$ defined by
\[Q(x,y)=tr(x\circ y)+tr(y\circ x) \]
for $(x,y)\in I_{p,q}$. Denote by $Z(Q)$ the zero locus of $Q$  in $I_{p,q}(F)$. Then $\mathcal{N}_{p,q}\subset Z(Q)\subset I_{p,q}$. Recall the following homogeneity result on tempered generalized functions.
\begin{thm}\label{duke}
	\cite[Theorem 5.1.7]{dima2009duke}	Let $L$ be a non-zero subspace of $\mathscr{C}_{Z(Q)}(I_{p,q})$ such that for every $f\in L$, one has that $\mathfrak{F}(f)\in L$ and $(\psi\circ Q)\cdot f\in L$ for all unitary character $\psi$ of $F$. Then $L$ is a completely reducible $F^\times$-subrepresentation of ${\mathscr{C}}(I_{p,q})$, and it has an eigenvalue of the form 
	\[
|-|^{\frac{1}{2}\dim I_{p,q}}\kappa^{-1}
	 \]
	where $\kappa$ is a pseudo-algebraic character of $F^\times$.
\end{thm}

Now we are prepared to prove Proposition \ref{fourier}. The basic idea is due to Chen-Sun in \cite{sun2020}.
\begin{proof}[Proof of Proposition \ref{fourier}]
	Denote by $L$ the space of all tempered generalized functions $f$ on $I_{p,q}$ with the properties in Proposition \ref{fourier}. Assume by contradiction that $L$ is nonzero. Then by Lemma \ref{key:lem} and Theorem \ref{duke}, one has
	\[|-|^{tr(2-\mathbf{h})|_{I_{p,q}^\mathbf{f}}}\kappa_1=\eta^2=|-|^{\dim I_{p,q}}\kappa_2^{-2}
 \]
	where $\kappa_1$ and $\kappa_2$ are two pseudo-algebraic characters of $F^\times$.
%	which is possible only when $p\neq q$ due to \cite[Lemma 3.12]{sun2020}.
This finishes the proof.
	%Then  is also supported on $\mathcal{O}$ due to Lemma \ref{lem:intertwin}.
	% We want to show that $\mathcal{O}\times\mathcal{O}$ is not weakly-coisotropic in $I_{p,q}\times I_{p,q}$, which contradicts the fact that the wavefront set of $f$ is weakly-coisotropic. Then $f=0$.
	% 
	 %It is easy to see that the origin $0$ does not belongs to the orbit $\mathcal{O}$. Hence $\{(0,0)\}\notin \mathcal{O}\times\mathcal{O}$ and so $\mathcal{O}\times\mathcal{O}$ is not weakly $I_{p,q}\times I_{p,q}$-coisotropic.
\end{proof}

\subsection{Proof of Theorem \ref{vanishing:I}} In this subsection, we will give the proof of Theorem \ref{vanishing:I}. We need the following definition and lemmas.

\begin{defn}
	We fix a grading on $\mathfrak{sl}_2(F)$ given by $\mathbf{h}\in\mathfrak{sl}_2(F)_0$ and $\mathbf{e},\mathbf{f}\in\mathfrak{sl}_2(F)_1$ where $\{\mathbf{h},\mathbf{e},\mathbf{f} \}$ is the $\mathfrak{sl}_2$-triple defined in \eqref{sl}. A graded representation of $\mathfrak{sl}_2(F)$ is a representation of $\mathfrak{sl}_2(F)$ on a graded vector space $V=V_0\oplus V_1$ such that
	\[\mathfrak{sl}_2(F)_i(V_j)\subset V_{i+j} \]
	for $i,j\in\mathbb{Z}/2\mathbb{Z}$. Then $V_0$ (resp. $V_1$) is called the even (resp. odd) part of $V$.
\end{defn}

\begin{lem} Every irreducible graded representation of $\mathfrak{sl}_2(F)$ is irreducible (as a usual representation of $\mathfrak{sl}_2(F)$).
\end{lem}
Denote by $V_\lambda^\omega$ the irreducible graded representation of $\mathfrak{sl}_2(F)$ with highest weight $\lambda$ and highest weight vector of parity $\omega\in\mathbb{Z}/2\mathbb{Z}$. Let $V=V_0\oplus V_1$ such that $\dim V_0=p$ and $\dim V_1=q=p+1$. Consider $V$ as a graded representation of $\mathfrak{sl}_2(F)$.
\begin{lem}\label{regular}
	If $V=V_0\oplus V_1$ is irreducible as a graded representation of $\mathfrak{sl}_2(F)$, then $\mathbf{e}$ is regular nilpotent, i.e., $\dim \mathcal{O}$ is the biggest dimension among the nilpotent orbits of $I_{p,p+1}$.
\end{lem}
 In general, there is a decomposition of $\mathfrak{sl}_2(F)$-graded modules 
\[V=V_{\lambda_1}^{\omega_1}\oplus V_{\lambda_2}^{\omega_2}\oplus\cdots\oplus V_{\lambda_d}^{\omega_d} \]
		for $d\geq1$. (See \cite{sun2020}.) There is an isomorphism
\[I_{p,q}=\Hom(V_0,V_1)\oplus\Hom(V_1,V_0)\cong\Hom(V_0\oplus V_1,V_0\oplus V_1)_1 \]
of $F$-vector spaces,
where $\Hom(V,V)_1$ is the odd part of $\Hom(V,V)$ as a graded $\mathfrak{sl}_2(F)$-module.

\begin{lem}\cite[Lemma 3.1]{sun2020}
Let $$m_{i,j}:=tr(2-\mathbf{h})|_{\Hom(V_{\lambda_i}^{\omega_i},V_{\lambda_j}^{\omega_j})_1^\mathbf{f}}+tr(2-\mathbf{h})|_{\Hom(V_{\lambda_j}^{\omega_j},V_{\lambda_i}^{\omega_i})_1^\mathbf{f} }-\dim V_{\lambda_i}^{\omega_i}\dim V_{\lambda_j}^{\omega_j}$$
for $i,j\in\{1,2,\cdots,d\}$. Then $$tr(2-\mathbf{h})|_{I_{p,q}^\mathbf{f}}-\dim I_{p,q}=\frac{1}{2}\sum_{\substack{1\leq i\leq d\\ 
1\leq j\leq d} }m_{i,j}+\frac{1}{2}(p-q)^2 $$ and
\[m_{i,j}=\begin{cases}
\min\{\lambda_i,\lambda_j \}+1,&\mbox{if }\lambda_i\not\equiv\lambda_j\pmod{2};\\
2\min\{\lambda_i,\lambda_j \}+2,&\mbox{if }\lambda_i\equiv \lambda_j\equiv1\pmod{2}\mbox{ and }\omega_i=\omega_j;\\
0,&\mbox{if }\lambda_i\equiv\lambda_j\equiv1\pmod{2}\mbox{ and }\omega_i\neq\omega_j;\\
-|\lambda_i-\lambda_j|-1,&\mbox{if }\lambda_i\equiv\lambda_j\equiv0\pmod{2}\mbox{ and }\omega_i=\omega_j;
\\
\lambda_i+\lambda_j+3,&\mbox{if }\lambda_i\equiv\lambda_j\equiv0\pmod{2}\mbox{ and }\omega_i\neq\omega_j.
\end{cases} \]\label{mij}
\end{lem}
\begin{prop}\label{observation}If $tr(2-\mathbf{h})|_{I_{p,p+1}^\mathbf{f}}=2p(p+1)$, 
%	\begin{itemize}
	%	\item 
%	either $tr(2-\mathbf{h})|_{I_{p,p+1}^\mathbf{f}}=2p(p+1)$,
%%	\end{itemize}
	 then there exists an $h\in H_{p,p+1}$ such that $\sigma\cdot\mathbf{e}=h\mathbf{e}h^{-1}$.
\end{prop}
\begin{proof} 	Suppose $V=\oplus_{i=1}^d V_{\lambda_i}^{\omega_{i}}$ with $d\geq 1$. If $\lambda_i$ is odd, then
	$\dim V_{\lambda_i}^{\omega_{i}}\cap V_0=\dim V_{\lambda_i}^{\omega_{i}}\cap V_1$. If $\lambda_i$ is even, then $\dim V_{\lambda_i}^{\omega_{i}}\cap V_0=\dim V_{\lambda_i}^{\omega_{i}}\cap V_1+(-1)^{\omega_{i}}$. Since $\dim V_1=\dim V_0+1$, we obtain that the number of indices $i$
	such that $\lambda_i$ is even and $\omega_{i}=1$ minutes the number of indices $i$ such that $\lambda_i$ is even and $\omega_{i}=0$ equals  $1$. Denote by $t$ the number of indices $i$ such that
	$\lambda_i$ is even and $\omega_{i}=0$.
%	We separate the proof into three cases.
%	\begin{itemize}
	%	\item 
	Assume that $tr(2-\mathbf{h})|_{I_{p,p+1}^\mathbf{f}}=2p(p+1)$.
	It is easy to see that $\mathbf{e}=\begin{pmatrix}
	0&\mathbf{1}_p&0\\ 0&0&0\\ \mathbf{1}_p&0&0
	\end{pmatrix}$ and $h=\begin{pmatrix}
	\omega_p&\\& \omega_{p+1}
	\end{pmatrix}$ where $\omega_1=\begin{pmatrix}
	1
	\end{pmatrix}$ is the $1\times 1$ matrix and $\omega_{i+1}:=\begin{pmatrix}
	0&1\\\omega_{i}&0
	\end{pmatrix}$ for $i=1,2,\cdots, p$. In this case, $V$ is irreducible as a $\mathfrak{sl}_2(F)$-graded representation.
	\par
	 In general, if $V$ is reducible, then $(q-p)^2+\sum_{\substack{1\leq i\leq d\\ 1\leq j\leq d}} m_{i,j}>0$. The following proof is similar to \cite[Lemma 7.7.5]{dima2009duke}.
	 Reorder the space $V_{\lambda_i}^{\omega_{i}}$ so that $\omega_{i}=0$ for $1\leq i\leq t$
 and $\omega_{i}=1$ for $i>t$. Furthermore, we require that $\lambda_1\geq\lambda_2\geq \cdots\geq \lambda_t$ and $\lambda_{t+1}\geq\lambda_{t+2}\geq\cdots\geq\lambda_{2t+1}=\lambda_d$. Then
 \begin{equation*}
 \begin{split}(q-p)^2+
 \sum_{\substack{1\leq i\leq d\\ 1\leq j\leq d}}m_{i,j}=&1+\sum_{\substack{1\leq i\leq t\\ 1\leq j\leq t}}(-|\lambda_i-\lambda_j|-1)+\sum_{\substack{t+1\leq i\leq 2t+1\\ 1\leq j\leq t}} (\lambda_i+\lambda_j+3)\\
& +\sum_{\substack{1\leq i\leq t\\ t+1\leq j\leq 2t+1}}(\lambda_i+\lambda_j+3)+\sum_{\substack{t+1\leq i\leq 2t+1\\ t+1\leq j\leq 2t+1}}(-|\lambda_i-\lambda_j|-1)\\
=&4t(t+1)-\sum_{\substack{1\leq i\leq t\\ 1\leq j\leq t}}|\lambda_i-\lambda_j|+2\sum_{\substack{t+1\leq i\leq 2t+1\\ 1\leq j\leq t}} (\lambda_i+\lambda_j)
 -\sum_{\substack{t+1\leq i\leq 2t+1\\ t+1\leq j\leq 2t+1}}|\lambda_i-\lambda_j| \\
 =&4t(t+1)+4\sum_{i=1}^t(\lambda_i+\lambda_{t+1+i}) i
 \end{split}
 \end{equation*}
 which is positive unless $t=0$. 
 \par
 If there is an another $\mathfrak{sl}_2(F)$-triple $\{\mathbf{h}',\mathbf{e}',\mathbf{f}' \}$ such that $V$ is irreducible, then Lemma \ref{regular} implies that both $\mathbf{e}$ and $\mathbf{e}' $ are regular nilpotent and so they are $H_{p,p+1}$-conjugate due to Kostant-Rallis' result that the regular nilpotent elements are in the same $H_{p,p+1}$-orbit (see \cite[Theorem 6]{kostant}).
	This finishes the proof.
\end{proof}
\begin{rem}\label{rem:fail}
	The above proposition does not hold for general $p$ and $q$. For instance, let $(p,q)=(6,8)$. There does exist a unipotent element $\mathbf{e}=(x,y)$ such that 
	\[\sum_{\substack{1\leq i\leq d\\ 1\leq j\leq d}}m_{i,j}+4=0 \]
	where $d=3, V=V_1^1\oplus V_2^1\oplus V_8^1$ is the decomposition of $\mathfrak{sl}_2(F)$-graded modules, $rank(x)=6$ and $rank(y)=5$. Therefore $\sigma\cdot\mathbf{e}\notin H_{6,8}\mathbf{e}$.
\end{rem}

Finally, we can give a proof of Theorem \ref{vanishing:I}. 
\begin{proof}[Proof of Theorem \ref{vanishing:I}] 
	It suffices to show that ${\mathscr{C}}_{\mathcal{N}_{p,q}}(I_{p,q})^{\tilde{H}_{p,q},\chi}=0$. 
		Due to Proposition \ref{fourier},  assume that $tr(2-\mathbf{h})|_{I_{p,q}^\mathbf{f}}=2pq$. 
	Suppose that $f\in \mathscr{C}_{\mathcal{O}}(I_{p,q})^{\tilde{H}_{p,q},\chi}$
	is a tempered generalized function on $I_{p,q}$ supported on the orbit $\mathcal{O}=H_{p,q}\mathbf{e}\subset\mathcal{N}_{p,q}$. Then its Fourier transform $\mathfrak{F}(f)$ is supported on $\mathcal{O}$ due to Lemma \ref{lem:intertwin}. Thanks to Proposition \ref{observation}
	\[ \sigma\cdot\mathbf{e}\in\mathcal{O}, \]
	 it implies that $f=0$ which means that every element in ${\mathscr{C}}_{\mathcal{N}_{p,q}}(I_{p,q})^{\tilde{H}_{p,q},\chi}$ is zero, as required.
	 
	 If $F$ is archimedean  and $\kappa_1\kappa_2^2=|-|^{2pq- tr(2-\mathbf{h})|_{I_{p,q}^{\mathbf{f}}}},$ then $\kappa_1=\kappa_2=\mathbf{1}$ and $tr(2-\mathbf{h})|_{I_{p,q}^\mathbf{f}}=2pq$.  Otherwise, it contradicts Lemma \ref{mij}. However, $\sigma\cdot\mathbf{e}\in H_{p,q}\mathbf{e}$ in this case. Thus $f=0$. This finishes the proof.
	% Let $q=p+1$. Similarly we can show that $\sigma\cdot \mathbf{e}\in\mathcal{O}$.
\end{proof}

\section{Proof of Theorem \ref{thm:A}} Recall that $n=p+q$ and $q=p+1$.
Let $\mathcal{H}_{p,q}:=H_{p,q}\times H_{p,q}$ be a  reductive group. Define
\[\tilde{\mathcal{H}}_{p,q}:=\mathcal{H}_{p,q}\rtimes\langle\sigma\rangle \]
where $\sigma$
acts on $\mathcal{H}_{p,q}$ by the involution $(h_1,h_2)\mapsto ((h_2^{-1})^t,(h_1^{-1})^t)$. Let $\tilde{\mathcal{H}}_{p,q}$ act on $\GL_n(F)$ by
\[(h_1,h_2)\cdot g=h_1gh_2^{-1} \]
and $\sigma\cdot g=g^t$ for $h_i\in H_{p,q}$ and $g\in\GL_n(F)$.
Let $\chi$ be the sign character of $\tilde{\mathcal{H}}_{p,q}$. Let $\mu\otimes\mu^{-1}$ be a character of $\mathcal{H}_{p,q}$.
 Let $\tilde{\mu}$
be a character of $\tilde{\mathcal{H}}_{p,q}$ twisted by the sign character, i.e.
$\tilde{\mu}|_{\mathcal{H}_{p,q}}=\mu\otimes\mu^{-1}$ and $\tilde{\mu}(\sigma)=-1$.

This section is devoted to  a proof of the following theorem.
\begin{thm}
	\label{thm:dist} Assume that $\mu$ is a good character of $H_{p,q}$. We have
	\[{\mathscr{C}}(\GL_n(F))^{\tilde{\mathcal{H}}_{p,q},\tilde{\mu}}=0. \]
\end{thm}
Then Theorem \ref{thm:A} will follow from Theorem \ref{thm:dist} immediately.

Suppose that 
\begin{equation}\label{x_k}
x_{p,k}=\begin{pmatrix}
&&\mathbf{1}_k\\& \mathbf{1}_{p-k}\\\mathbf{1}_k\\&&&\mathbf{1}_{p+1-k}
\end{pmatrix}
\end{equation} $k=0,1,\cdots,p$. Then the orbits
$\tilde{\mathcal{H}}_{p,q} x_{p,k}$ are closed orbits in $\GL_n(F)$ (see \cite[Proposition 4.1]{jacquet1996linear}). 
\begin{lem}\label{Jac:lem}
	\cite[proposition 4.1]{jacquet1996linear} The following double cosets
	$${H}_{p,p+1}\begin{pmatrix}
	g_{11}&0&g_{12}\\0&x_{p-\nu,k}&0\\ g_{21}&0&g_{22}
	\end{pmatrix}H_{p,p+1}$$
	exhaust all closed orbits 
	 in $H_{p,p+1}\backslash \GL_{2p+1}(F)/H_{p,p+1}$, where $x_{p-\nu,k}$ (for $\nu=0,1,\cdots,p-k$) is defined in \eqref{x_k}, $g=\begin{pmatrix}
	 g_{11}&g_{12}\\g_{21}&g_{22}
	 \end{pmatrix}$ satisfies
	$$g\begin{pmatrix}
	\mathbf{1}_\nu\\&-\mathbf{1}_\nu
	\end{pmatrix}g^{-1}
	\begin{pmatrix}
	\mathbf{1}_\nu\\&-\mathbf{1}_\nu
	\end{pmatrix}=\begin{pmatrix}
	A&\mathbf{1}_\nu\\ A^2-\mathbf{1}_\nu&A
	\end{pmatrix}$$  and $A\in Mat_{\nu,\nu}(F)$ is a semisimple matrix without eigenvalues $\pm1$.
\end{lem}
\begin{proof}
	See \cite[Theorem 4.13]{carmeli2015stability}.
\end{proof}

Thanks to Theorem \ref{thm2.1}, if 
\[\mathscr{C}(R(N_{\mathcal{O},x}^{\GL_n(F)}))^{\tilde{\mathcal{H}}_{p,q,x},\tilde{\mu}}=0 \]
implies
\[\mathscr{C}(Q(N_{\mathcal{O},x}^{\GL_n(F)}))^{\tilde{\mathcal{H}}_{p,q,x},\tilde{\mu}}=0 \]
for any $\tilde{\mathcal{H}}_{p,q}$-closed orbit $\mathcal{O}=\tilde{\mathcal{H}}_{p,q}x$, where $\tilde{\mathcal{H}}_{p,q,x}$ is the stabilizer of $x$, then Theorem \ref{thm:dist} holds.

At first, let us consider the simple case: $\nu=0$.
\begin{lem}\label{nu=0}
	We have 
	$${\mathscr{C}}(R(N^{\GL_n(F)}_{\tilde{\mathcal{H}}_{p,q} x_{p,k},x_{p,k} }))^{\tilde{\mathcal{H}}_{p,q,x_{p.k}},\tilde{\mu}}=0\Longrightarrow
	{\mathscr{C}}(Q(N^{\GL_n(F)}_{\tilde{\mathcal{H}}_{p,q} x_{p,k},x_{p,k} }))^{\tilde{\mathcal{H}}_{p,q,x_{p,k}},\tilde{\mu}}=0 $$
	where $\tilde{\mathcal{H}}_{p,q,x_{p,k}}=\{h\in\tilde{\mathcal{H}}_{p,q}|h\cdot x_{p,k}=x_{p,k} \}$ is the stabilizer of $x_{p,k}$ in $\tilde{\mathcal{H}}_{p,q}$.
\end{lem}
\begin{proof}
	By easy computation, we have $\tilde{\mathcal{H}}_{p,q,x_{p,k}}\cong(\GL_k(F)\times\GL_k(F) \times {H}_{p-k,q-k})\rtimes \langle\sigma\rangle$, where $\sigma$ acts on $\GL_k(F)\times \GL_k(F)\times H_{p-k,q-k}$ by the involution
	\[(g_1,g_2,h)\mapsto ((g_2^{-1})^t,(g_1^{-1})^t,(h^{-1})^t) \]
	for $g_i\in\GL_k(F)$ and $h\in H_{p,q}$.
	 The normal bundle (see \cite[Lemma 4.3]{sun2020}) is given by
	\[N_{\tilde{\mathcal{H}}_{p,q}x_{p,k},x_{p,k}}^{\GL_n(F)}= \frac{\mathfrak{gl}_n(F)}{\mathfrak{h}_{p,q}+Ad_{x_{p,k}}\mathfrak{h}_{p,q} }\cong I_{k,k}\oplus I_{p-k,q-k}. \]
	The action of $\tilde{\mathcal{H}}_{p,q,x_{p,k}}$ on $N_{\tilde{\mathcal{H}}_{p,q}x_{p,k},x_{p,k}}^{\GL_n(F)}$ is given by
	\[(g_1,g_2,h)\cdot (x,y,z)=(g_2xg_1^{-1},g_1yg_2^{-1},hzh^{-1}) \]
	and $\sigma\cdot (x,y,z)=(x^t,y^t,z^t)$ for $g_i\in\GL_k(F),h\in H_{p-k,q-k}, (x,y)\in I_{k,k}$ and $z\in I_{p-k,q-k}$.
	By \cite[Proposition 2.5.8]{dima2009duke}, ${\mathscr{C}}(N_{\tilde{\mathcal{H}}_{p,q}x_{p,k},x_{p,k}}^{\GL_n(F)} )^{\tilde{\mathcal{H}}_{p,q,x_{p,k}},\tilde{\mu}}$
	is a product of
	\[{\mathscr{C}}(I_{k,k})^{(\GL_k(F)\times\GL_k(F))\rtimes\langle\sigma\rangle,\widetilde{\mu_k\otimes\mu_k^{-1}}}\mbox{  and  } {\mathscr{C}}(I_{p-k,q-k})^{\tilde{H}_{p-k,q-k},\chi} \]
	where $\mu_k=\mu_F\circ\det$ is the character of $\GL_k(F)$, $\mu_k\otimes\mu_k^{-1}$
	is the character of $\GL_k(F)\times\GL_k(F)$, $\widetilde{\mu_k\otimes\mu_k^{-1}}$ is the character of $(\GL_k(F)\times\GL_k(F))\rtimes\langle\sigma\rangle$ twisted by the sign character.
	%Define $N_k=\{g\in _:g^k=0 \}$.
	Thanks to Theorem \ref{vanishing:I}, ${\mathscr{C}}_{\mathcal{N}_{p-k,q-k}}(I_{p-k,q-k})^{\tilde{\mathcal{H}}_{p-k,q-k},\chi}=0$.
Thus it suffices to show that $${\mathscr{C}}_{\mathcal{N}_k}(I_{k,k})^{(\GL_k(F)\times\GL_k(F))\rtimes\langle\sigma\rangle,\widetilde{\mu_k\otimes\mu_k^{-1}}}=0.$$
It follows from \cite[Proposition 3.9]{sun2020} since $\mu_F$ is a good character.
 We have finished the proof.
\end{proof}
\begin{rem}
	If $x\in H_{p,q}$, then $\mathcal{O}=\tilde{\mathcal{H}}_{p,q}x=H_{p,q}$ is a closed orbit in $\GL_n(F) $.  The group embedding from the stabilizer subgroup $\tilde{\mathcal{H}}_{p,q,x}\cong\tilde{H}_{p,q}$ of $x$ to $\tilde{\mathcal{H}}_{p,q}$ is given by
	\[(h,\delta)\mapsto \begin{cases}
	(h,x^{-1}hx)&\mbox{ if }\delta=1;
	\\ (h,x^{-1}hx^t)\sigma&\mbox{ if }\delta=\sigma.
	\end{cases} \]
(See \cite[Page 12]{sun2020}.)	Similarly, we have $$\mathscr{C}(R(N_{\mathcal{O},x}^{\GL_n(F)}))^{\tilde{\mathcal{H}}_{p,q,x},\tilde{\mu}}=0
\Longrightarrow \mathscr{C}(Q(N_{\mathcal{O},x}^{\GL_n(F)}))^{\tilde{\mathcal{H}}_{p,q,x},\tilde{\mu}}=0
$$ for $x\in H_{p,q}$.
\end{rem}
Now we can give a proof of Theorem \ref{thm:dist}.
\begin{proof}[Proof of Theorem \ref{thm:dist} ] Applying Theorem \ref{thm2.1}, we only need to prove that there does not exist any $(\tilde{\mathcal{H}}_{p,q},\tilde{\mu})$-equivariant tempered generalized function on the normal bundle of $\tilde{\mathcal{H}}_{p,q}$-closed orbits.
	Thanks to Lemma \ref{nu=0}, we have proved  Theorem \ref{thm:dist} if $\nu=0$. Thus applying Lemma \ref{Jac:lem}, it is reduced to prove that 
	\begin{equation}\label{substep}
	\mathscr{C}(R(N_{\tilde{\mathcal{H}}_{p,q}x,x}^{\GL_{n}(F)}))^{\tilde{\mathcal{H}}_{p,q,x},\tilde{\mu}}=0\Longrightarrow\mathscr{C}(Q(N_{\tilde{\mathcal{H}}_{p,q}x,x}^{\GL_{n}(F)}))^{\tilde{\mathcal{H}}_{p,q,x},\tilde{\mu}}=0  \end{equation}
	for the closed orbit $\tilde{\mathcal{H}}_{p,q}x$, where
	\[x=\begin{pmatrix}
	g_{11}&&g_{12}\\&x_{p-\nu,k}&\\ g_{21}&&g_{22}
	\end{pmatrix}\in\GL_{n}(F) \] and $x\theta_{p,q}(x^{-1})=\begin{pmatrix}
	A&&&&&\mathbf{1}_\nu\\&-\mathbf{1}_k\\&&\mathbf{1}_{p-\nu-k}\\&&&-\mathbf{1}_k\\&&&&\mathbf{1}_{q-\nu-k}
	&\\A^2-\mathbf{1}_\nu&&&&&A
	\end{pmatrix}$,
	 $A$ is a semisimple element in $Mat_{\nu,\nu}(F)$ for $\nu=1,2,\cdots, p-k,$ without eigenvalues $\pm1$.
	Futhermore, we may assume that $x\theta_{p,q}(x)=\theta_{p,q}(x)x$ (where $x$ is called normal in the sense of \cite[\S7.4]{dima2009duke}), $A$ is a scalar matrix and $A^2\neq\mathbf{1}_\nu$. Then 
	\[\tilde{\mathcal{H}}_{p,q,x}\cong(\GL_{\nu}(F)\times(\GL_k(F)\times\GL_k(F))\times H_{p-\nu-k,q-\nu-k})\rtimes\langle\sigma\rangle\cong (\GL_{\nu}(F)\rtimes\langle\sigma\rangle)\times \tilde{\mathcal{H}}_{p-\nu,q-\nu,x_{p-\nu,k}}\]
	$\tilde{\mu}|_{(\GL_{\nu}(F)\rtimes\langle\sigma\rangle)}=\chi$ is the sign character
	and $$N_{\tilde{\mathcal{H}}_{p,q}x,x}^{\GL_n(F)}\cong Mat_{\nu,\nu}(F)\oplus I_{k,k}\oplus I_{p-k-\nu,q-k-\nu} \cong Mat_{\nu,\nu}(F)\oplus N_{\tilde{\mathcal{H}}_{p-\nu,q-\nu}x_{p-\nu,k},x_{p-\nu,k}}^{\GL_{n-2\nu}(F)},$$ where $\GL_{\nu}(F)$ acts on $Mat_{\nu,\nu}(F)$ by inner conjugation and $\sigma$ acts on $Mat_{\nu,\nu}(F)$ by the matrix transpose. Therefore \eqref{substep} follows from Lemma \ref{nu=0} and
	\[\mathscr{C}(Mat_{\nu,\nu}(F))^{\GL_\nu(F)\rtimes\langle\sigma\rangle,\chi }=0. \]
(See \cite[Theorem D]{sun2020}.)	This finishes the proof.
\end{proof}

\section{Proof of Theorem \ref{p=1}}
The method in this paper does not work for arbitrary $p$ and $q$ (see Remark \ref{rem:fail}). However, we can still prove several cases if $p$ is small, such as $p=1$. The main purpose in this section
is to study the case for the pair $(\GL_n(F),\GL_1(F)\times \GL_{n-1}(F)  )$. Recall that $H_{1,n-1}=\GL_1(F)\times \GL_{n-1}(F)$. We can define $I_{1,n-1},\mathcal{N}_{1,n-1},\tilde{H}_{1,n-1},\mathcal{H}_{1,n-1}$ and
$\tilde{\mathcal{H}}_{1,n-1}$ similarly. Given a closed orbit $\tilde{\mathcal{H}}_{1,n-1}x$ in $\GL_n(F)$, we denote $\tilde{\mathcal{H}}_{1,n-1,x}$ the stabilizer of $x$ in $\tilde{\mathcal{H}}_{1,n-1}$.

We follow the method in the previous section to give a proof of Theorem \ref{p=1}
\begin{proof}[Proof of Thoerem \ref{p=1}] The case for $n=2$ is trivial. Assume that $n\geq3$.
	Applying Theorem \ref{thm2.1} and Lemma \ref{Jac:lem}, we only need to prove that
	\[\mathscr{C}(R(N_{\tilde{\mathcal{H}}_{1,n-1}x,x}^{\GL_n(F)}))^{\tilde{\mathcal{H}}_{1,n-1,x},\chi}=0 
	\Longrightarrow \mathscr{C}(Q(N_{\tilde{\mathcal{H}}_{1,n-1}x,x}^{\GL_n(F)}))^{\tilde{\mathcal{H}}_{1,n-1,x},\chi}=0
	\]
	for $x=\begin{pmatrix}
	&1\\1\\&& \mathbf{1}_{n-2}
	\end{pmatrix}^k(k=0,1)$ or $x$ satisfying \[x\omega_{1,n-1}x^{-1}\omega_{1,n-1}=\begin{pmatrix}
	A&&1\\&\mathbf{1}_{n-2}\\A^2-1&&A
	\end{pmatrix} \] where $A$ is a scalar in $Mat_{\nu,\nu}(F)=F$ and $A^2\neq1$. Now we separate them into three cases.
	\begin{itemize}
		\item Assume $\nu=0$ and $k=0$. Then $N_{\tilde{\mathcal{H}}_{1,n-1}x,x}^{\GL_n(F)}\cong I_{1,n-1}$ and the stabilizer of $x=\mathbf{1}_n$ is isomorphic to $\tilde{H}_{1,n-1}$. Then it is enough to show that
		\[\mathscr{C}_{\mathcal{N}_{1,n-1}}(I_{1,n-1})^{\tilde{H}_{1,n-1},\chi}=0. \]
		In fact, we will prove that a stronger result
		\[ \mathscr{C}_{\mathcal{N}_{1,n-1}}(I_{1,n-1})^{\tilde{H}_{1,n-2},\chi}=0. \]
		(See the equality \eqref{NIH} which will be proved later.) Then we are done.
%		Take a nonzero element $\mathbf{e}\in \mathcal{N}_{1,n-1}$ and formulate a $\mathfrak{sl}_2(F)$-triple
%		$\{\mathbf{e},\mathbf{f},\mathbf{h} \}$ (see \eqref{sl}). Suppose that the generalized function $f$ is $(\tilde{H}_{1,n-1},\chi)$-equivariant and its support $supp(f)$ contains $\mathbf{e}$. We will prove that $f=0$. Since $\mathbf{e}$ is arbitrary, it implies that $$\mathscr{C}_{\mathcal{N}_{1,n-1}}(I_{1,n-1})^{\tilde{H}_{1,n-1},\chi}=0.$$
%		If $n=2$ or $3$, then $$\mathscr{C}_{\mathcal{N}_{1,n-1}}(I_{1,n-1})^{\tilde{H}_{1,n-1},\chi}=0$$ due to \cite[Proposition 3.9]{sun2020} and Theorem \ref{vanishing:I}. Thus $f=0$.

		\item Assume $\nu=0$ and $k=1$. Then $N_{\tilde{\mathcal{H}}_{1,n-1}x,x}^{\GL_n(F)}\cong F\oplus F$ and
		\[\tilde{\mathcal{H}}_{1,n-1,x}\cong (\GL_1(F)\times\GL_1(F)\times\GL_{n-2}(F) )\rtimes\langle\sigma\rangle. \]
		The action of $\tilde{\mathcal{H}}_{1,n-1,x}$ on $F\oplus F$ is given by
		$$(g_1,g_2,h)\cdot(x,y)=(g_2xg_1^{-1},g_1yg_2^{-1})$$ 
		and $\sigma\cdot(x,y)=(x,y)$ for $g_i\in\GL_1(F),h\in\GL_{n-2}(F)$ and $x,y\in F$.
		Moreover, $$ \mathscr{C}_{\mathcal{N}_{1,1}}(I_{1,1})^{\tilde{\mathcal{H}}_{1,n-1,x},\widetilde{\mu_F\otimes\mu_F^{-1}}}=0$$
 since the  element $\sigma$ fixes $\mathcal{N}_{1,1}$ pointwisely.
		\item Assume $k=0$ and $\nu=1$. Then $N_{\tilde{\mathcal{H}}_{1,n-1}x,x}^{\GL_n(F)}\cong F$ and
		$\tilde{\mathcal{H}}_{1,n-1,x}\cong (\GL_1(F)\times\GL_{n-2}(F))\rtimes\langle\sigma\rangle$. The action on $F$ is trivial. This implies $\mathscr{C}(F)^{\tilde{\mathcal{H}}_{1,n-1,x},\chi}=0$.
	\end{itemize}
We have finished the proof of Theorem \ref{p=1}.
\end{proof}

\section{Applications}
In this section, we use a similar idea to give   another application in the representation theory.

In \cite{maxmirabolic}, assuming that $F$ is non-archimedean, Gurevich investigated the role of the mirabolic subgroup on  the symmetric variety $\GL_n(F)/H_{p,n-p}$ where $H_{p,n-p}=\GL_p(F)\times\GL_{n-p}(F)$. More precisely,
let $P$ be a mirabolic subgroup of $\GL_n(F)$ consisting of matrices with last row vector $(0,\cdots,0,1)$.  Let $\GL_n(F)$ act on $Mat_{n,n}(F)$ by inner conjugation.
Bernstein \cite{bernstein1984} proved that any $P$-invariant generalized function on $Mat_{n,n}(F)$ must be $\GL_n(F)$-invariant. We expect that there is a more general phenomenon related to the mirabolic subgroup $P$.

Define $I_{p,n-p}$ and $\mathcal{N}_{p,n-p}$ as before. Let $P\cap H_{p,n-p}$ act on $I_{p,n-p}$ by inner conjugation. It is expected that any $P\cap H_{p,n-p}$-invariant tempered generalized function on $I_{p,n-p}$ supported on $\mathcal{N}_{p,n-p}$ is also $H_{p,n-p}$-invariant.
The following is a baby case.
\begin{lem}
	Any $P\cap H_{n-1,1}$-invariant tempered generalized function on $I_{n-1,1}$ supported on $\mathcal{N}_{n-1,1}$ is also $H_{n-1,1}$-invariant.
\end{lem}
\begin{proof}
	Note that $H_{n-1,1}=\GL_{n-1}(F)\times\GL_1(F) $ and $P\cap H_{n-1,1}\cong \GL_{n-1}(F)$. Let $f\in\mathscr{C}_{\mathcal{N}_{n-1,1}}(I_{n-1,1})^{P\cap H_{n-1,1}}$. Given arbitrary
	$h=\begin{pmatrix}
	a\\&b
	\end{pmatrix}\in  H_{n-1,1}$ for $a\in\GL_{n-1}(F) $ and $b\in F^\times$,
	\[f(h\cdot(x,y))=f(a^{-1}xb,b^{-1}ya)=f(ba^{-1}x,yab^{-1} )=f(x,y) \]
	for any $(x,y)\in \mathcal{N}_{n-1,1}$. Thus $f$ is $H_{n-1,1}$-invariant.
\end{proof}
 Gurevich  proved that
any $P\cap H_{1,n-1}$-invariant generalized function on $\mathcal{N}_{1,n-1}$ is also $H_{1,n-1}$-invariant (see \cite[Theorem 4.2]{maxmirabolic}). Then by \cite[Theorem 3.9]{maxmirabolic} and \cite[Corollary 5.1]{maxmirabolic}, he proved that any $P\cap H_{1,n-1}$-invariant linear functional on an $H_{1,n-1}$-distinguished irreducible smooth representation of $\GL_n(F)$ is also $H_{1,n-1}$-invariant (see \cite[Theorem 1.1]{maxmirabolic}). 
We will give a new and shorter  proof to \cite[Theorem 4.2]{maxmirabolic} here, including the archimedean case.
\begin{prop} \label{prop:max}
	 Let $P$ be the standard mirabolic subgroup of $\GL_n(F)$ consisting of matrices with last row vector $(0,\cdots,0,1)$.
	Let $I_{p,n-p},\mathcal{N}_{p,n-p}$ and $\mathscr{C}_{\mathcal{N}_{p,n-p}}(I_{p,n-p})$ be as before. Then
	\[\mathscr{C}_{\mathcal{N}_{1,n-1}}(I_{1,n-1})^{P\cap H_{1,n-1}}=\mathscr{C}_{\mathcal{N}_{1,n-1}}(I_{1,n-1})^{H_{1,n-1}}. \]
\end{prop}
\begin{proof} Assume $n-1\geq2$.
	We will prove that any generalized function $f\in \mathscr{C}_{\mathcal{N}_{1,n-1}}(I_{1,n-1})^{P\cap H_{1,n-1}}$ satisfies $f(x)=f(x^t)$ for all $x\in I_{1,n-1}$. Then $f$ is invariant with respect to $P^t\cap H^t_{1,n-1}$ and so $f$ is invariant under $$\langle P\cap H_{1,n-1},P^t\cap H^t_{1,n-1}\rangle=H_{1,n-1}.$$
	It is known that $H_{1,n-2}$ is a proper subgroup in $P\cap H_{1,n-1}$. Let $\tilde{H}_{1,n-1}$ be as usual and $\chi$ be its sign character. We will show that 
	\begin{equation}\label{NIH}
	\mathscr{C}_{\mathcal{N}_{1,n-1}}(I_{1,n-1})^{\tilde{H}_{1,n-2},\chi}=0. \end{equation}
	Note that $I_{1,n-1}=I_{1,n-2}\oplus V\oplus V^\ast$ with $\dim V=1$. Let $(\mathbf{e},v,v^\ast)\in I_{1,n-2}\oplus V\oplus V^\ast$ be a unipotent element in $\mathcal{N}_{1,n-1}$. Then
	$v^\ast(v)=0$ (see \cite[\S6.1]{aizenbud2013partial}). Thus either $v=0$ or $v^\ast=0$. 
	Without loss of generality, assume $v=0$. 
	Take any $f\in\mathscr{C}_{\mathcal{N}_{1,n-1}}(I_{1,n-2}\oplus V\oplus V^\ast)^{\tilde{H}_{1,n-2},\chi}$ such that
	$(\mathbf{e},0,v^\ast)\in supp(f)$. 
	Then the partial Fourier transform $\mathfrak{F}_{V\times V^\ast}(f)$
is	also supported on $\tilde{H}_{1,n-2}(\mathbf{e},0,v^\ast)$; see \cite[\S4.2]{aizenbud2013partial}. 
%which contradicts the fact (see \cite[Appendix C]{sun2012mult} or \cite[Lemma 6.3.4]{aizenbud2013partial}) that $\mathfrak{F}_{V\times V^\ast}(f)$ must be supported on $\mathcal{N}_{1,n-2}\times V\times\{0\}$.
%Assume that $v^\ast\neq0$. Note that $V^\ast\setminus\{0\}\cong F^\times$ which is isomorphic to an $F^\times\times F^\times$-manifold
Thanks to \cite[Lemma 6.3.4]{aizenbud2013partial} that $(I_{1,n-2}\oplus V\oplus \{0\})\sqcup (I_{1,n-2}\oplus\{0\}\oplus V^\ast)$ does not support any nonzero $H_{1,n-2}$-invariant generalized functions, 
%that there does not exist any $F^\times$-invariant generalized function supported on the cross $$\{(v,v^\ast)\in V\times V^\ast| v^\ast(v)=0\}\setminus\{(0,0)\},$$
%Then $\sigma\cdot(\mathbf{e},0,v^\ast)\notin supp(f)$ and so
%\[f(\mathbf{e},0,v^\ast)=-f(\sigma\cdot (\mathbf{e},0,v^\ast))=0. \]
% we have $supp(f)=\tilde{H}_{1,n-2}\cdot(\mathbf{e},0,0)=\tilde{H}_{1,n-2}\mathbf{e}\times\{(0,0)\}$.
%applying Frobenius reciprocity (see \cite[\S1.5]{bernstein1984})
%\[\mathscr{C}_{\{0\}\times V^\ast\setminus\{(0,0)\} }(V\oplus V^\ast)^{\tilde{H}_{1,n-2},\chi}\cong  \]
  $f=0$. Then any $\tilde{H}_{1,n-2}$-invariant generalized function on $I_{1,n-1}$ is invariant under transposition. This finishes the proof.
\end{proof}

Proposition \ref{prop:max} implies that any $P\cap H_{1,n-1}$-invariant linear functional on an $H_{1,n-1}$-distinguished irreducible smooth admissible representation of $\GL_{n}(F)$ is also $H_{1,n-1}$-invariant. 
\begin{coro}
	Any $P\cap H_{1,n-1}$-invariant linear functional on an $H_{1,n-1}$-distinguished irreducible smooth admissible representation of $\GL_n(F)$ is also $H_{1,n-1}$-invariant.
\end{coro}
\begin{proof}
		Denote by $\mathscr{D}(\GL_{n}(F))$
	the distributions on $\GL_{n}(F)$.
	Following \cite[Corollary 5.1]{maxmirabolic}, 
	and \cite[Theorem 3.6]{kemarsky}, it suffices to show that
	\[ \mathscr{D}(\GL_{n}(F))^{(P\cap H_{1,n-1})\times H_{1,n-1}}=\mathscr{D}(\GL_n(F))^{H_{1,n-1}\times H_{1,n-1}}. \]
	Note that $\mathscr{D}(\GL_{n}(F))^{(P\cap H_{1,n-1})\times H_{1,n-1}}\cong\mathscr{D}(\GL_{n}(F)/H_{1,n-1})^{P\cap H_{1,n-1}}$; see \cite[Lemma 3.7]{kemarsky}.
	Thus it is equivalent to proving
	\[\mathscr{D}(\GL_{n}(F)/H_{1,n-1})^{P\cap H_{1,n-1}}=\mathscr{D}(\GL_{n}(F)/H_{1,n-1})^{H_{1,n-1}}. \]
		 We shall show that 
	\[\mathscr{C}(\GL_{n}(F)/H_{1,n-1})^{\tilde{H}_{1,n-2},\chi}=0. \]
		which will imply $\mathscr{D}(\GL_{n}(F)/H_{1,n-1})^{\tilde{H}_{1,n-2},\chi}=0$
	due to a general principle of ”distribution versus Schwartz distribution” (see \cite[Theorem 4.0.2]{dima2009duke}). Then we are done since $P\cap H_{1,n-1}$ and its transpose generate the whole group $H_{1,n-1}$.
	
	Suppose that $n\geq3$. Applying Theorem \ref{thm2.1} and Lemma \ref{Jac:lem}, it is enough to show that $$\mathscr{C}_{\mathcal{N}_{1,1}}(I_{1,1})^{\tilde{H}_{1,1},\chi}=0=\mathscr{C}_{\mathcal{N}_{1,n-1}}(I_{1,n-1})^{\tilde{H}_{1,n-2},\chi}.$$
	which follows from \cite[Theorem D]{sun2020} and \eqref{NIH}. 
	Here the action of $\tilde{H}_{1,1}$
	on $I_{1,1}$ is given by
	\[(a,b)\cdot (x,y)=(bxa^{-1},ayb^{-1})\mbox{  and  }\sigma(x,y)=(x,y). \]
	This finishes the proof.
\end{proof}

In fact, we can prove a bit more.
Let $P$ be the standard mirabolic subgroup of $\GL_{2p+1}(F)$ 
consisting of matrices with last row vector $(0,\cdots,0,1)$. Then
$H_{p,p}$ is a proper subgroup of $P\cap H_{p,p+1}$.
\begin{thm} \label{mir} Let $F$ be nonarchimedean.
	Any $P\cap H_{p,p+1}$-invariant linear functional on an $H_{p,p+1}$-distinguished irreducible smooth  representation of $\GL_{2p+1}(F)$ is also $H_{p,p+1}$-invariant.
\end{thm}
Before we give the proof of Theorem \ref{mir}, we shall apply Theorem \ref{mir} to study the relation between the exterior square $L$-function and the $H_{p,p+1}$-distinguished spherical representation $\pi$.

Let $F$ be a finite field extension of $\mathbb{Q}_p$. Let $O_F$ be the ring of integers of $F$. Let $\varpi$ be the uniformizer in $O_F$ and $|\varpi|=q^{-1}$ where $q$ is the cardinality of the residue field $O_F/\varpi O_F$ of $F$. 
Let $B(F)$ be the standard Borel subgroup of $\GL_{2p+1}(F)$ with unipotent radical $N(F)$. Let $\pi$ be a unitary spherical principal series representation of $\GL_{2p+1}(F)$ distinguished by $H_{p,p+1}$ such that
\[\pi=Ind_{B}^{\GL_{2p+1}(F)}(\chi_1\otimes\cdots\otimes\chi_{2p+1})(\mbox{normalized induction}) \]
where $\chi_i$ are unitary characters of $F^\times$. Let $K$ be a maximal open compact subgroup of $\GL_{2p+1}(F)$.
\begin{lem}\cite{shintani}
	Up to a scalar there exists a unique right $K$-invariant Whittaker function $W$, in $\pi$
	given by $W(\varpi^{\boldsymbol{\lambda}})=0$
	unless $\boldsymbol{\lambda}=(\lambda_1,\lambda_2,\cdots,\lambda_{2p+1})$ in $\mathbb{Z}^{2p+1}$ satisfies $\lambda_1\geq\lambda_2\geq\cdots\geq\lambda_{2p+1}$, where 
	\[W(\varpi^{\boldsymbol{\lambda}})=\delta_B^{1/2}(\varpi^{\boldsymbol{\lambda}})\frac{\det((\chi_i(\varpi^{\lambda_j+2p+1-j}))_{i,j})}{\det((\chi_i(\varpi^{2p+1-j}))_{i,j})} \]
	and $\delta_B$ is the modular function of $B(F)$.
\end{lem}
Let $H'_{p,p+1}$ be the image of $H_{p,p+1}$ in $\GL_{2p+1}(F)$ under the following embedding
\[\begin{pmatrix}
	(a_{i,j})\\&(b_{i,j})
\end{pmatrix}\mapsto \begin{pmatrix}
c_{i,j}
\end{pmatrix}\in \GL_{2p+1}(F) \]
where $(a_{i,j})\in\GL_p(F)$, $(b_{i,j})\in\GL_{p+1}(F)$ and $c_{i,j}=\begin{cases}
	b_{s,t}&\mbox{ if }i=2s-1,j=2t-1;\\
	a_{s,t}&\mbox{ if }i=2s,j=2t;\\
	0&\mbox{ otherwise.}
\end{cases}$
\par
Consider the integral
\[W\mapsto \int_{N(F)\cap H'_{p,p+1}\backslash P\cap H'_{p,p+1} }W(p)|\det(p)|^sdp.  \]
It will give us a $P\cap H'_{p,p+1}$-invariant linear functional on $\pi$ when $s=0$ which is $H'_{p,p+1}$-invariant as well due to Theorem \ref{mir}. Moreover, the integral is related to the exterior square $L$-function $L(s,\pi,\Lambda^2)$ which sets up a relation between $L(s,\pi,\Lambda^2)$ and the distinction problem of $\pi$. Denote by $L(s,\pi)$ the standard $L$-function of $\pi$. Take a measure $dp$ on $P$ such that the volume of the compact subset $P\cap H'_{p,p+1}\cap K$ is $1$.
\begin{thm}
	Let $\pi=Ind_{B(F)}^{\GL_{2p+1}(F)}(\chi_1\otimes\cdots\otimes\chi_{2p+1})$ be a unitary spherical representation of $\GL_{2p+1}(F)$. If $\pi$ is distinguished by $H'_{p,p+1}$, then $L(s,\pi,\Lambda^2)L(s,\pi)$
	has a pole at $s=0$.
\end{thm}
\begin{proof} Suppose that $B(F)=AN(F)$ where $A$ is the split torus and so $A\subset H'_{p,p+1}$. Then $P\cap H'_{p,p+1}=(N(F)\cap H'_{p,p+1})(P\cap A)(P\cap H'_{p,p+1}\cap K)$. Denote by $\delta_{P\cap H'_{p,p+1}}$ the modular character of $(N(F)\cap H'_{p,p+1})(P\cap A)$.
	Note that $\delta_B(\varpi^{\boldsymbol{\lambda}})^{1/2}=\delta_{P\cap H'_{p,p+1}}(\varpi^{\boldsymbol{\lambda}}) $. Then
	\[\int_{N(F)\cap H'_{p,p+1}\backslash P\cap H'_{p,p+1} }W(p)|\det(p)|^sdp=\sum_{\lambda_1\geq \cdots\geq\lambda_{2p}\geq\lambda_{2p+1}=0 }\frac{\det((\chi_i(\varpi^{\lambda_j+2p+1-j}))_{i,j})}{\det((\chi_i(\varpi^{2p+1-j}))_{i,j})}(q^{-s})^{tr\boldsymbol{\lambda}}  \]
	which equals $$\Big(1-q^{-(2p+1)s}\prod_i\chi_i(\varpi)\Big)\cdot\prod_{i<j}(1-\chi_i(\varpi)\chi_j(\varpi)q^{-2s}) ^{-1}\cdot\prod_i(1-\chi_i(\varpi)q^{-s})^{-1}.$$ (See \cite[\S1.5 Example 4]{macdonald}.)
	Since $\pi$ is distinguished by $H_{p,p+1}'$, $\prod_i\chi_i$ is trivial. Thus
	$$\lim_{s\rightarrow0} (1-q^{-(2p+1)s})L(s,\pi,\Lambda^2)L(s,\pi)\neq0.$$ Therefore $L(0,\pi,\Lambda^2)L(0,\pi)=\infty$.
\end{proof}

\subsection{Proof of Theorem \ref{mir}}
This subsection is devoted to a proof of Theorem \ref{mir}. The basic ideas come from \cite{sun2012mult,sun2020}.  The generalized function $f$ will be restricted to a smaller open subset which can be handled easily. It will give us a very strict condition for the support of $f$.  Then we will show that any $H_{p,p}$-invariant tempered generalized function on $I_{p,p+1}$ is invariant under transposition, which will imply Theorem \ref{mir}.

From Proposition \ref{prop:max}, we have seen that
\[I_{p,p+1}=I_{p,p}\oplus V\oplus V^\ast \]
where $V\oplus V^\ast$ is equipped with a natural non-degenerate quadratic form
%\[\langle (v_1,v_1^\ast),(v_2,v_2^\ast)\rangle\mapsto v_1^\ast(v_2)-v_2^\ast(v_1) \]
%for $v_i\in V$ and $v_i^\ast\in V^\ast$.
\[(v,v^\ast)\mapsto v^\ast(v) \]
for $v\in V$ and $v^\ast\in V^\ast$, which induces a symmetric bilinear form $\langle -,-\rangle$.
Let $F^\times$ act on $\mathscr{C}_{\mathcal{N}_{p,p+1}}(I_{p,p+1})$ by
\[t\cdot f(x,y,v,v^\ast)=f(t^{-1}x,t^{-1}y,t^{-1}v,t^{-1}v^\ast) \]
for $(x,y)\in I_{p,p},v\in V$ and $v^\ast\in V^\ast$.
Recall that
\[\begin{pmatrix}
a&0\\0&b
\end{pmatrix}\cdot f(x,y,v,v^\ast)=f(a^{-1}xb,b^{-1}ya,a^{-1}v,v^\ast a ) \]
for $a,b\in\GL_p(F)$ and
\[\sigma\cdot f(x,y,v,v^\ast)=f(y^t,x^t,(v^\ast)^t,v^t). \]
	Let $(\mathbf{e},v_0,v_0^\ast)\in \mathcal{N}_{p,p+1}$ and $\mathfrak{O}=(H_{p,p}\times F^\times)(\mathbf{e},v_0,v_0^\ast)$ be a $H_{p,p}\times F^\times$-orbit in $\mathcal{N}_{p,p+1}$. Then $\mathbf{e}=\begin{pmatrix}
0&x_0\\y_0&0
\end{pmatrix}\in \mathcal{N}_{p,p}$ and $v_0^\ast(x_0\circ y_0)^kv_0=0$ for any non-negative integer $k$. (See \cite[\S6.1]{aizenbud2013partial}.)

	 Let $\{\mathbf{h},\mathbf{e},\mathbf{f}\}$ be a graded $\mathfrak{sl}_2(F)$-triple (see \eqref{sl}) related to $I_{p,p}$, which  integrates to an algebraic homomorphism
\[ \mathrm{SL}_2(F)\longrightarrow \GL_{2p}(F). \] 
Denote by $D_t$ the image of $\begin{pmatrix}
t\\&t^{-1}
\end{pmatrix}$ in $H_{p,p}$.
Let $$T:=\{(D_t,t^{-2})\in H_{p,p}\times F^\times|t\in F^\times \}$$ be a closed subgroup in $H_{p,p}\times F^\times$ which
fixes the element $\mathbf{e}$. Define
 \begin{equation*}%\label{E(e)}
  E(\mathbf{e}):=\Big\{(v,v^\ast)\in V\times V^\ast\big|\begin{matrix} 
%  \langle h\cdot(v,v^\ast),(v,v^\ast)\rangle=0\mbox{ for any }h\in H_{p,p} \mbox{ satisfying that }h\cdot\mathbf{e}\\ \mbox{ is a multiple of }\mathbf{e} 
% \mbox{ and }
 v^\ast(x_0\circ y_0)^kv=0\mbox{ for all non-negative integers }k
 \end{matrix} \Big\}.
 \end{equation*}
 and  
 \[V(\mathbf{e}):=\{(v,v^\ast)\in E(\mathbf{e})|\langle h\cdot(v,v^\ast),(v,v^\ast)\rangle=0\mbox{ for any }h\in\langle D_t\rangle \}. \]

The  following lemma is similar to \cite[Lemma 3.13]{sun2020}.
\begin{lem}\label{sun:lem} 
	Let $\eta$ be an eigenvalue for the action of $F^\times$ on $\mathscr{C}_\mathfrak{O}(I_{p,p+1})^{H_{p,p}}$.  Then 
	\[\eta^2=
	%\begin{cases}
%|-|^{tr(2-\mathbf{h})|_{I_{p,p}^\mathbf{f}}+2p}	&\mbox{if }F \mbox{ is non-archimedean;}\\
|-|^{tr(2-\mathbf{h})|_{I_{p,p}^\mathbf{f}}+2p}\kappa_1\kappa_2^{-2}
%&\mbox{if }F\mbox{ is archimedean}.
	%\end{cases}
	 \]
	 for some  pseudo-algebraic characters $\kappa_1$ and $\kappa_2$  of $F^\times$.
%	for a certain eigenvalue $\kappa$ of the action of $T$ on $\mathscr{C}_{E(\mathbf{e})}(V\oplus V^\ast)$.
\end{lem}
\begin{proof}
	Consider the map
	\begin{equation}
	\label{submersion}
	H_{p,p}\times F^\times\times (I_{p,p}^\mathbf{f}\oplus V\oplus V^\ast)\longrightarrow I_{p,p}\oplus V\oplus V^\ast 
	\end{equation}
	via $(h,\xi,v,v^\ast)\mapsto h.(\mathbf{e}+\xi+v+v^\ast)$ for $\xi\in I_{p,p}^\mathbf{f},h\in H_{p,p}\times F^\times,v\in V$ and $v^\ast\in V^\ast$, which is
	submersive at every point of $H_{p,p}\times F^\times\times\{(0,v_0,v_0^\ast)\}$. Moreover,  $H_{p,p}\times F^\times\times\{(0,v_0,v_0^\ast)\}$ is open in the inverse image of $\mathfrak{O}=(H_{p,p}\times F^\times).(\mathbf{e},v_0,v_0^\ast)$ under the map \eqref{submersion}. (See \cite[Page 18]{sun2020}.)
	Thanks to \cite[Lemma 2.7]{jiang2011trans},
	the restriction map yields an injective linear map
	\[ \mathscr{C}_{\mathfrak{O}}(I_{p,p}\oplus V\oplus V^\ast)^{H_{p,p}\times F^\times,\mathbf{1}\times\eta}\rightarrow
	\mathscr{C}_{\{0\}\times E(\mathbf{e})}(I_{p,p}^\mathbf{f}\oplus V\oplus V^\ast)^{T,\mathbf{1}\times\eta|_T} \]
	where $\mathbf{1}\times\eta|_{T}((D_t,t^{-2}))=\eta(t)^{-2}$.
	%$$E(\mathbf{e})=\{(v,v^\ast)\in V\times V^\ast|v^\ast(x_0\circ y_0)^kv=0\mbox{ for all non-negative integers }k \}.$$
	It is easy to see that the representation $\mathscr{C}_{\{0\}}(I_{p,p}^\mathbf{f})$
	of $T$ is completely reducible and every eigenvalue has the form
	\[ (D_t,t^{-2})\mapsto |t|^{tr(\mathbf{h}-2)|_{I_{p,p}^\mathbf{f}}}\kappa_1(t)^{-1}\]
where $\kappa_1$ is a pseudo-algebraic character of $F^\times$. Thus $$\eta(t)^2=|t|^{tr(2-\mathbf{h})|_{I_{p,p}^\mathbf{f}}}\kappa_1(t)\eta_0^{-1}(t)$$ for any $t\in F^\times$, where $\eta_0$ is an eigenvalue for the action of $T$ on $\mathscr{C}_{E(\mathbf{e})}(V\oplus V^\ast)$.
  In order to compute $\eta_0$, we will restrict $\eta_0$ to a smaller subspace $\mathscr{C}_{V(\mathbf{e})}(V\oplus V^\ast)$ of $\mathscr{C}_{E(\mathbf{e})}(V\oplus V^\ast)$. 
 \par
 Define a symplectic form on $(V\oplus V^\ast)\times (V\oplus V^\ast)$ as follow
 \[ < (x_1,y_1),(x_2,y_2)>:=\langle x_1,y_2\rangle-\langle y_1,x_2\rangle   \]
 where $x_i,y_i\in V\oplus V^\ast$. Then $V\oplus V^\ast$ is a maximal isotropic subspace.
 Consider the Weil representation on $\mathrm{Mp}_{4p}(F)=\mathrm{Mp}((V\oplus V^\ast)\times(V\oplus V^\ast),<-,->)$.
 % $\SL_2(F)\times\Oo_{2p}(F)\hookrightarrow\Sp_{2p}(F) \times\Oo_{2p}(F)$. Here $S(V\oplus V^\ast)$
 %	is diagonally embedded into $$S((V\oplus V^\ast)^p)\cong\underbrace{S(V\oplus V^\ast)\otimes S(V\oplus V^\ast)\otimes\cdots\otimes S(V\oplus V^\ast)}_{\text{$p$ times}}.$$
 Under the Weil representation $\omega_\psi$,
 \[\begin{cases}
 \omega_\psi\begin{pmatrix}
 A\\& (A^t)^{-1}
 \end{pmatrix}\varphi(x)=|\det A|^{1/2}\varphi(A^{-1}x),&\mbox{ for }A\in\GL_{2p}(F),\\
 \omega_\psi\begin{pmatrix}
 \mathbf{1}_{2p}&N\\&\mathbf{1}_{2p}
 \end{pmatrix}\varphi(x)=\psi(\langle Nx,x\rangle)\varphi(x),&\mbox{ for }N=N^t\in Mat_{2p,2p}(F),
 \end{cases}
 \]
 for $\varphi\in S(V\oplus V^\ast)$ and $x\in V\oplus V^\ast$. We may extend $\omega_\psi$ from the Schwartz  space $S(V\oplus V^\ast)$ to the generalized function space $\mathscr{C}(V\oplus V^\ast)$.
 Note that
 \[\begin{pmatrix}
 X\\& X^{-1}
 \end{pmatrix}=\begin{pmatrix}
 \mathbf{1}_n&-X\\&\mathbf{1}_n
 \end{pmatrix}\begin{pmatrix}
 \mathbf{1}_n\\ X^{-1}&\mathbf{1}_n
 \end{pmatrix}\begin{pmatrix}
 \mathbf{1}_n&\mathbf{1}_n-X\\&\mathbf{1}_n
 \end{pmatrix}\begin{pmatrix}
 \mathbf{1}_n\\-\mathbf{1}_n&\mathbf{1}_n
 \end{pmatrix}\begin{pmatrix}
 \mathbf{1}_n&\mathbf{1}_n
 \\&\mathbf{1}_n	\end{pmatrix}  \]
 holds for any $X\in\GL_{n}(F)$. Here we only need  the case that $X$ is a diagonal matrix. 
 Denote
 $D_t=\begin{pmatrix}
 A_t\\&B_t
 \end{pmatrix}$ and $X_t=\begin{pmatrix}
 A_t\\& A_t^{-1}
 \end{pmatrix}$. 
 Then the action of $D_t$ on $V\oplus V^\ast$ is given by
 $$(v,v^\ast)\mapsto (A_tv,v^\ast A_t^{-1}).$$ 
 % which keeps the quadratic form on $V\oplus V^\ast$ and the symplectic form $<-,->$. 
It is obvious that
 %\[V(\mathbf{e})=\{(v_1,v_2,\cdots,v_p;v_1^\ast,\cdots v_p^\ast)\in V\times V^\ast|v_i^\ast v_i=0 \mbox{ for all }i \}\subset E(\mathbf{e}). \]
 %\[V(\mathbf{e}):=\{(v,v^\ast)\in E(\mathbf{e})|\langle h\cdot(v,v^\ast),(v,v^\ast) \rangle =0\mbox{ for all }h\in H_{\mathbf{e}} \} \]
 %where $H_\mathbf{e}=\{h\in H_{p,p}:h\cdot\mathbf{e}=\lambda\mathbf{e}\mbox{ for }\lambda\in F^\times \}$.
% Note that
 \begin{equation*}
 \begin{split}\omega_\psi
 \begin{pmatrix}
 \mathbf{1}_{2p}& X_t\\&\mathbf{1}_{2p}
 \end{pmatrix}f(v,v^\ast)&=\psi(\langle (A_tv,v^\ast A_t^{-1}),(v,v^\ast)\rangle)f(v,v^\ast)
 \\ 
 &=f(v,v^\ast)
 \end{split}
 \end{equation*}
 for any $f\in\mathscr{C}_{V(\mathbf{e})}(V\oplus V^\ast)$. Then
 $\begin{pmatrix}
 \mathbf{1}_{2p}&X_t\\ &\mathbf{1}_{2p}
 \end{pmatrix}$  acts on $\mathscr{C}_{V(\mathbf{e})}(V\oplus V^\ast)$
 trivially and so is $\begin{pmatrix}
 \mathbf{1}_{2p}\\ X_t^{-1}&\mathbf{1}_{2p}
 \end{pmatrix}$. 
 %We may arrange $A_t$ such that $\det A_t=1$.
 Thus $D_t$ does not contribute to $\eta_0$. Therefore $\eta_0$ has the form
 \[(D_t,t^{-2})\mapsto |t^{-2}|^{\frac{1}{2}\dim(V\oplus V^\ast)}\kappa_2^{-1}(t^{-2})=|t|^{-2p}\kappa_2(t)^2 \]
 and so $\eta(t)^2=|t|^{tr(2-\mathbf{h})|_{I_{p,p}^\mathbf{f}}}\kappa_1(t)\cdot |t|^{2p}\kappa_2(t)^{-2}$ for  some pseudo-algebraic characters $\kappa_i$ of $F^\times$.
\end{proof}
Let $\mathbf{e}$ be a nilpotent element in $I_{p,p}$. Let $\{\mathbf{h},\mathbf{e},\mathbf{f} \}$ be the
$\mathfrak{sl}_2(F)$-triple (see \eqref{sl}). Then Chen-Sun \cite[Lemma 3.12]{sun2020} proved
\begin{equation}\label{lem:inequality}
2p^2< tr(2-\mathbf{h})|_{I_{p,p}^\mathbf{f}}\leq 4p^2. \end{equation}

Now we can give a proof of Theorem \ref{mir}
\begin{proof}[Proof of Theorem \ref{mir}] 
	%Denote by $\mathscr{D}(\GL_{2p+1}(F))$
	%the distributions on $\GL_{2p+1}(F)$.
	Following \cite[Corollary 5.1]{maxmirabolic}, 
	%and \cite[Theorem 3.6]{kemarsky}, 
	it suffices to show that
	\begin{equation}\label{6.4}
	\mathscr{C}(\GL_{2p+1}(F))^{(P\cap H_{p,p+1})\times H_{p,p+1}}=\mathscr{C}(\GL_{2p+1}(F))^{H_{p,p+1}\times H_{p,p+1}}. \end{equation}
	Note that $\mathscr{C}(\GL_{2p+1}(F)/H_{p,p+1})^{P\cap H_{p,p+1}}\cong \mathscr{C}(\GL_{2p+1}(F))^{(P\cap H_{p,p+1})\times H_{p,p+1}}$. 
	%see \cite[Lemma 3.7]{kemarsky}. 
	Thus \eqref{6.4} is equivalent to 
	\begin{equation}
	\label{max}
	\mathscr{C}(\GL_{2p+1}(F)/H_{p,p+1})^{P\cap H_{p,p+1}}=\mathscr{C}(\GL_{2p+1}(F)/H_{p,p+1})^{H_{p,p+1}}.
	\end{equation}
	 We shall show that 
	\[\mathscr{C}(\GL_{2p+1}(F)/H_{p,p+1})^{\tilde{H}_{p,p},\chi}=0. \]
%	which will imply $\mathscr{D}(\GL_{2p+1}(F)/H_{p,p+1})^{\tilde{H}_{p,p},\chi}=0$
	%due to a general principle of ”distribution versus Schwartz distribution” (see \cite[Theorem 4.0.2]{dima2009duke}).
	Then the identity \eqref{max} follows from the fact
	that $P\cap H_{p,p+1}$ and its transposition generate the whole group $H_{p,p+1}$.
	Applying Theorem \ref{thm2.1} and Lemma \ref{Jac:lem}, it is enough to show that
	\begin{equation}\label{6.6}
		\mathscr{C}_{\mathcal{N}_{p,p+1}}(I_{p,p}\oplus V\oplus V^\ast)^{\tilde{H}_{p,p},\chi}=0 \end{equation}
	with $\dim V=\dim V^\ast=p$.
	%Thus it suffices to show that any $H_{p,p}$-invariant generalized function $f$ on $\Gamma(I_{p,p}\oplus V\oplus V^\ast)$ is also invariant under transposition. 
	
	Now the rest of this part is devoted to proving the equality \eqref{6.6}.
	%\[\mathscr{C}_{\mathcal{N}_{p,p+1}}(I_{p,p}\oplus V\oplus V^\ast)^{\tilde{H}_{p,p},\chi}=0. \]
	Take $(\mathbf{e},v_0,v_0^\ast)\in \mathcal{N}_{p,p+1}$ and the $\mathfrak{sl}_2(F)$-triple $\{\mathbf{h},\mathbf{e},\mathbf{f} \}$ related to $I_{p,p}$. Denote $\mathfrak{O}:=H_{p,p}(\mathbf{e},v_0,v_0^\ast)\subset \mathcal{N}_{p,p+1}$.
	Recall that $F^\times$ acts on $\mathscr{C}_{\mathfrak{O}}(I_{p,p}\oplus V\oplus V^\ast)$
	by
	\[t.f(x,y,v,v^\ast)=f(t^{-1}x,t^{-1}y,t^{-1}v,t^{-1}v^\ast) \]
	for $t\in F^\times,(x,y)\in I_{p,p}$ and $f\in \mathscr{C}_{\mathfrak{O}}(I_{p,p}\oplus V\oplus V^\ast)$.
	Let $\eta$ be an eigenvalue for the action of $F^\times$ on $\mathscr{C}_{\mathfrak{O}}(I_{p,p}\oplus V\oplus V^\ast)$. 
	%Suppose $F$ is non-archimedean.
	By Lemma \ref{sun:lem},
	$\eta^2=|-|^{tr(2-\mathbf{h})|_{I_{p,p}^\mathbf{f}}+2p}$.
	By Theorem \ref{duke}, one has $\eta^2=|-|^{\dim I_{p,p+1}}$ and so 
%	$\kappa$ must be of the form
%	\[(D_t,t^{-2})\mapsto |t|^{tr(2-\mathbf{h})|_{I_{p,p}^\mathbf{f}}-2p(p+1)}. \]
%	Recall that $\kappa$ is an eigenvalue of the action of $T$ on $\mathscr{C}_{E(\mathbf{e})}(V\oplus V^\ast)$.
	 $$tr(2-\mathbf{h})|_{I_{p,p}^\mathbf{f}}+2p=2p(p+1),$$
	which contradicts the inequality \eqref{lem:inequality}. 
%If $F=\mathbb{R}$ and $\eta^2=|-|^{tr(2-\mathbf{h})|_{I_{p,p}^\mathbf{f}}+a+2p-2b}$, then
%	\begin{equation}\label{refine:sun}
%	tr(2-\mathbf{h})|_{I_{p,p}^\mathbf{f}}+a-2b=2p^2-2c \end{equation}
%	with $c\geq0$.
%	Recall a refined version of \eqref{lem:inequality} that
%	\[tr(2-\mathbf{h})|_{I_{p,p}^\mathbf{f}}-2p^2=\frac{1}{2}\sum_{\substack{1\leq i\leq d\\ 1\leq j\leq d}}m_{ij}=2t^2+\sum_{i=1}^t(2t+1-2i)(\lambda_i+\lambda_{t+i}) \]
%where $m_{ij},\lambda_i$ are defined in Lemma \ref{mij} and $\dim V_0=\dim V_1$, $d=2t$ in this case; see the proof of Proposition \ref{observation}. There is only one  solution for \eqref{refine:sun}: $tr(2-\mathbf{h})|_{I_{p,p}^\mathbf{f}}-2p^2+a=2b-2c$,
%\[a=0,b=-2\mbox{ and }t=1,\lambda_1=\lambda_2=0. \]
%Therefore $p=t+\sum_{i=1}^t\lambda_i= 1$ and $n=3$. However, \eqref{NIH} implies that
%\[ \mathscr{C}_{\mathcal{N}_{1,2}}(I_{1,1}\oplus V\oplus V^\ast)^{\tilde{H}_{1,1},\chi}=0. \]
%	 Thus 
%	\[\mathscr{C}_{\mathcal{N}_{p,p+1}}(I_{p,p}\oplus V\oplus V^\ast)^{\tilde{H}_{p,p},\chi}=0. \]
	This finishes the proof.
\end{proof}
\begin{rem}
	It seems that our method fails for Theorem \ref{mir} when $F$ is archimedean. The reason is that there are many possible solutions for
	\[\eta^2=|-|^{\dim I_{p,p+1}}\kappa^{-2} \]
	due to Lemma \ref{sun:lem}.
\end{rem}

\subsection*{Acknowlegement} The author would like to thank Dmitry Gourevitch for useful discussions. The author also wants to thank Maxim Gurevich for explaining his thesis work \cite{maxmirabolic}. The author would also like to thank the anonymous referee for many useful comments which greatly improve
the exposition of this paper.
This work was partially supported by
the ERC, StG grant number 637912 and ISF grant 249/17.

\bibliographystyle{amsalpha}
\bibliography{linear}

\end{document}